\documentclass[12pt,reqno]{amsart}

\usepackage[latin1]{inputenc}
\usepackage{amsmath}
\usepackage{amsfonts}
\usepackage{amssymb}
\usepackage{graphics}
\usepackage{enumerate}
\usepackage{amssymb,amsmath,amsthm,amscd,latexsym,verbatim,graphicx,amsfonts}

\usepackage{amscd}
\usepackage{amsmath}
\usepackage{amssymb}
\usepackage{amsthm}
\usepackage{latexsym}
\usepackage{verbatim}

\theoremstyle{plain}
\newtheorem{theorem}{Theorem}[section]

\newtheorem{lemma}[theorem]{Lemma}

\theoremstyle{definition}

\theoremstyle{remark}

\newcommand{\bbZ}{\mathbb{Z}}
\newcommand{\bbR}{\mathbb{R}}
\newcommand{\bbC}{\mathbb{C}}
\newcommand{\bbQ}{\mathbb{Q}}

\newcommand{\bbN}{\mathbb{N}}

\DeclareMathOperator*{\supp}{supp}

\DeclareMathOperator*{\Imag}{Im}

\DeclareMathOperator*{\diag}{diag}

\title[]{Discrete $m$-functions with Doubly Palindromic Continued Fraction Coefficients}

\begin{document}
\maketitle

\begin{center}
\textbf{Hunter Handley and Brian Simanek\footnote{This author gratefully acknowledges support from the Simons Foundation through collaboration grant 707882.}}
\end{center}
\date{}

\begin{abstract}
We demonstrate that discrete $m$-functions with eventually periodic continued fraction coefficients have an algebraic relationship to their second solutions if and only if the periodic part of the sequence of continued fraction coefficients is doubly palindromic.  In this setting, doubly palindromic means that each sequence is a concatenation of two palindromes and a compatibility condition between the lengths of these palindromes is satisfied.
\end{abstract}

\vspace{4mm}

\footnotesize\noindent\textbf{Keywords:} Discrete $m$-functions, Continued Fractions, Transfer Matrices

\vspace{2mm}

\noindent\textbf{2020 Mathematics Subject Classification:} Primary 11A55; Secondary 42C05, 47B36

\vspace{2mm}

\normalsize

\section{Introduction}\label{intro}

Continued fractions are a ubiquitous tool in mathematics that has garnered significant attention in the last 100 years, especially since the breakthrough work of Ramanujan exposed their beauty and utility.  While the classical theory focused on continued fraction expansions of irrational real numbers, there is a parallel theory that concerns functions of a complex variable.  We will specifically consider the continued fraction expansions of discrete $m$-functions, which are Cauchy transforms of compactly supported probability measures on the real line, i.e.
\[
m(z)=\int\frac{d\rho(x)}{x-z},\qquad\qquad\supp(\rho)\subseteq[-t,t],\qquad\rho(\bbR)=1
\]
(see \cite[Theorem 2.3.6]{Rice}).  Such a function is known to have a continued fraction expansion that converges everywhere in $\bbC_+=\{z:\Imag[z]>0\}$.  In this case, the continued fraction coefficients (also called Jacobi parameters) come in the form of two bounded sequences $\{a_n\}_{n=1}^{\infty}$ and $\{b_n\}_{n=1}^{\infty}$, where each $a_n>0$ and each $b_n\in\bbR$.  With these coefficients, the continued fraction representation of $m(z)$ takes the following form:
\[
m(z)=\cfrac{1}{b_1-z-\cfrac{a_1^2}{b_2-z-\cfrac{a_2^2}{b_3-z-\cfrac{a_3^2}{b_4-z-a_4^2\cdots}}}}
\]
and this expression is valid for all $z\in\bbC_+$ (see \cite{Akh,Wall} or \cite[Equation 3.2.41]{Rice}),.  A converse to the above result follows from Favard's Theorem: if $\{a_n\}_{n=1}^{\infty}$ is a bounded sequence in $(0,\infty)$ and $\{b_n\}_{n=1}^{\infty}$ is a bounded sequence in $\bbR$, then there is a discrete $m$-function whose sequences of continued fraction coefficients are precisely these sequences.

It will be helpful for us to think of the sequences of continued fraction coefficients as a single sequence of pairs $\{(a_n,b_n)\}_{n=1}^{\infty}$.  We will be interested in the case when this sequence of pairs is eventually periodic.  By this we mean that we can express the sequence of coefficients in the following  form:
\begin{equation}\label{perpar}
\{(\alpha_1,\beta_1),(\alpha_2,\beta_2),\ldots,(\alpha_k,\beta_k),\overline{(a_1,b_1),\ldots,(a_p,b_p)}\},
\end{equation}
where the line over the last $p$ coefficients indicates that this string is infinitely repeated for the remainder of the sequence.  We will refer to the string $\{(\alpha_1,\beta_1),\ldots,(\alpha_k,\beta_k)\}$ as the non-periodic portion of the sequence and the rest of the sequence as the periodic portion of the sequence.  Notice that our division between the non-periodic portion of the sequence and the periodic portion of the sequence is somewhat arbitrary in that we can increase the first index of the periodic part if we so choose.  By making an appropriate choice for this division, we may arrange our notation so that
\begin{equation}\label{kp}
(\alpha_k,\beta_k)=(a_p,b_p),
\end{equation}
so \textit{we will always assume that this is the case}.

A straightforward calculation reveals that if $M$ is a discrete $m$-function with eventually periodic continued fraction coefficients, then $M$ is a quadratic irrationality, meaning there are polynomials $\alpha,\beta,\gamma\in\bbR[z]$ such that
\[
\alpha(z) M(z)^2+\beta(z) M(z)+\gamma(z)=0,\qquad\qquad z\in\bbC_+.
\]
By use of the quadratic formula, one can see that the quadratic equation
\[
\alpha(z) y^2+\beta(z) y+\gamma(z)=0
\]
also has a second solution that we will denote by $\widetilde{M}(z)$.  The function $\widetilde{M}(z)$ may not be a discrete $m$-function or have any simple relationship to a discrete $m$-function.

Much of what is known about the continued fraction expansions of discrete $m$-functions is analogous to what is known about continued fraction expansions of real numbers.  For example, the aforementioned property relating periodicity of the continued fraction coefficients to being the solution of a quadratic equation is the function theoretic analog of a result for real numbers known as Lagrange's Theorem (see \cite[Theorem 7.13]{Jungle}).  Our investigation is partly motivated by the recent paper \cite{DMS}, which provides a function theoretic analog of Serret's Theorem (concerning rational numbers whose continued fraction expansion is a finite palindrome).
We will be interested in the relationship between $M(z)$ and $\widetilde{M}(z)$ and to better understand this relationship, we look to the analogous result in the real number setting.  The following result was proven in \cite{Burger05}.

\begin{theorem}\label{burgers}
Suppose $\alpha\in\bbR\setminus\bbQ$ is a root of a degree two polynomial $P\in\bbZ[x]$ and let $\tilde{\alpha}$ be the other solution of $P$.  Then there exist $A,B,C,D\in\bbZ$ with $|AD-BC|=1$ so that
\[
\frac{1}{\tilde{\alpha}}=\frac{A\alpha+B}{C\alpha+D}
\]
if and only if the (eventually periodic) continued fraction coefficients of $\alpha$ have periodic part that can be written as the concatenation of two palindromes.
\end{theorem}

Our main theorem will be an analog of Theorem \ref{burgers} for discrete $m$-functions.  To properly state it, we must determine the proper condition on the sequences of continued fraction coefficients.  It will be necessary that each sequence has periodic part that can be expressed as a concatenation of two palindromes, but an additional compatibility condition on these sequences will also be required.

Consider two periodic sequences $\{a_j\}_{j=1}^{\infty}$ and $\{b_j\}_{j=1}^{\infty}$, which satisfy
\begin{equation}\label{perp}
a_{j+p}=a_j,\qquad b_{j+p}=b_j,\qquad\qquad j\in\bbN.
\end{equation}
We will say that the sequence $\{(a_j,b_j)\}_{j=1}^{\infty}$ is \textit{doubly palindromic} if there exist $\ell\in\bbN$ such that $\{a_1,\ldots,a_p\}$ is the concatenation of two palindromes with the first having length $\ell$ and the second having length $p-\ell$ and $\{b_1,\ldots,b_p\}$ is the concatenation of two palindromes with the first having length $\ell+1$ and the second having length $p-\ell-1$.  If necessary, we will describe such a sequence as being \textit{doubly palindromic with period $p$ and first length $\ell$}.

In order to ensure that each sequence is in fact the concatenation of two palindromes, we will need to assume that $\ell<p-1$.  We suffer no loss of generality in making this assumption because we can always interpret a $p$-periodic sequence as a $2p$-periodic sequence and this reinterpretation preserves the property of being doubly palindromic, as we see in the following example.

\medskip

\noindent\textbf{Example.} Consider the sequences formed by repeating the strings
\[
a_1a_2a_2a_1\tilde{a}_1\qquad\mbox{and}\qquad b_1b_2b_3b_2b_1,
\]
which would be doubly palindromic with $p=5$ and $\ell=4$ if we allowed for palindromes of length zero in our definition of doubly palindromic.  However, if we instead consider the strings of length $10$:
\[
a_1a_2a_2a_1\tilde{a}_1a_1a_2a_2a_1\tilde{a}_1\qquad\mbox{and}\qquad b_1b_2b_3b_2b_1b_1b_2b_3b_2b_1,
\]
Then this pair is doubly palindromic with $p=10$ and $\ell=4$.

\medskip

From the above example, we see how we can avoid setting $\ell=p-1$ in our definition of doubly palindromic.  We have already stated that we will always assume that the periodic portion of the sequence \eqref{perpar} will be chosen so that the last non-periodic coefficients are equal to the last coefficients of one period of the periodic coefficients.  This can easily be arranged by adding one full period of the periodic portion of the coefficients to the non-periodic coefficients.  Thus, for a doubly palindromic sequence satisfying \eqref{perp}, we may assume that $p$ is such that each of $\{a_1,\ldots,a_p\}$ and $\{b_1,\ldots,b_p\}$ is a concatenation of two non-empty palindromes and that $(\alpha_k,\beta_k)=(a_p,b_p)$ in \eqref{perpar}.

The statement of our main theorem will also make use of orthogonal polynomials.  While there are many interpretations of such polynomials, we will think of them as a sequence of polynomials generated using the continued fraction coefficients $\{a_n\}_{n=1}^{\infty}$ and $\{b_n\}_{n=1}^{\infty}$.  Indeed, for any such sequences of coefficients, we will set $p_0(z)=1$ and then define $p_{n+1}(z)$ from $\{p_j(z)\}_{j=0}^n$ by the formula
\[
zp_n(z)=a_{n+1}p_{n+1}(z)+b_{n+1}p_n(z)+a_{n}p_{n-1}(z),\qquad\qquad n\in\bbN.
\]
(we set $p_{-1}=0$).  For us, it will be important to note the dependence of $p_n(z)$ on the sequences $\{a_n\}_{n=1}^{\infty}$ and $\{b_n\}_{n=1}^{\infty}$.  Thus, we will often write $p_n(z;\{a_j,b_j\}_{j=1}^n)$ for the degree $n$ polynomial defined as above.  We will also use the notation $q_n$ to denote the second kind polynomial, which is given by
\[
q_n(z;\{a_j,b_j\}_{j=1}^n)=\frac{1}{a_1}p_{n-1}(z;\{a_j,b_j\}_{j=2}^n),\qquad n\in\bbN
\]
and $q_0(z)=0$ and $q_{-1}(z)=-1$ (see \cite[Theorem 3.2.2]{Rice}).

The last piece of notation that we will need involves fractional linear transformations.  For a matrix
\[
T=\begin{pmatrix} a & b\\ c & d\end{pmatrix}
\]
we define
\[
f_T(z)=\frac{az+b}{cz+d}.
\]
With this notation, it is straightforward to check that $f_A(f_B(z))=f_{AB}(z)$.

Now we are ready to state our main result.

\begin{theorem}\label{mainthm}
Suppose $M$ is a discrete $m$-function with eventually periodic continued fraction coefficients as in \eqref{perpar} and satisfying \eqref{kp}.  Let $\widetilde{M}$ be the corresponding second solution.  The periodic part of the sequence \eqref{perpar} is doubly palindromic with period $p$ and first length $\ell$ if and only if
\begin{equation}\label{maineq}
\frac{1}{\alpha_k^2\widetilde{M}(z)}
=f_{T_3T_2T_1}(M(z)),
\end{equation}
where (with $\alpha_0=\alpha_k$)
\begin{align*}
T_1&=\begin{pmatrix}
p_k(z;\{\alpha_j,\beta_j\}_{j=1}^k) & q_k(z;\{\alpha_j,\beta_j\}_{j=1}^k)\\
-\alpha_kp_{k-1}(z;\{\alpha_j,\beta_j\}_{j=1}^k) & -\alpha_kq_{k-1}(z;\{\alpha_j,\beta_j\}_{j=1}^k)
\end{pmatrix}\\
T_2&=\begin{pmatrix}
p_{\ell+1}(z;\{a_j,b_j\}_{j=1}^{\ell+1}) & q_{\ell+1}(z;\{a_j,b_j\}_{j=1}^{\ell+1})\\
-a_{\ell+1}p_{\ell}(z;\{a_j,b_j\}_{j=1}^{\ell+1}) & -a_{\ell+1}q_{\ell}(z;\{a_j,b_j\}_{j=1}^{\ell+1})
\end{pmatrix}\\
T_3&=\begin{pmatrix}
p_{k}(z;\{\alpha_{k-j},\beta_{k-j+1}\}_{j=1}^{k}) & q_{k}(z;\{\alpha_{k-j},\beta_{k-j+1}\}_{j=1}^{k})\\
-\alpha_{k}p_{k-1}(z;\{\alpha_{k-j},\beta_{k-j+1}\}_{j=1}^{k}) & -\alpha_{k}q_{k-1}(z;\{\alpha_{k-j},\beta_{k-j+1}\}_{j=1}^{k})
\end{pmatrix}
\end{align*}
\end{theorem}

\noindent\textit{Remark.} Each of the matrices $T_j$ for $j=1,2,3$ is a transfer matrix (see \cite[Equation 3.2.19]{Rice}) conjugated by the $2\times2$ matrix $\diag\{1,-1\}$.  We will call such matrices \textit{conjugated transfer matrices}.  From the definition of transfer matrix, it follows that the product $T_3T_2T_1$ is also a conjugated transfer matrix.

\medskip

The next section is devoted to the proof of Theorem \ref{mainthm}. A natural question that follows from Theorem \ref{mainthm} is if one can prove a similar result with the roles of $M$ and $\widetilde{M}$ reversed.  In the case of irrational numbers, the answer to that question follows easily from Theorem \ref{burgers}.  However, in the case of discrete $m$-functions, we will see in Section \ref{further} that the situation is not symmetric and the reverse relationship between $M$ and $\widetilde{M}$ only occurs if the continued fraction coefficients of $M$ are purely periodic.  


\section{Proof of Theorem \ref{mainthm}}\label{mainproof}

In this proof, we will make use of a technique called coefficient stripping.  If $s$ is a discrete $m$-function with continued fraction coefficients $\{(a_n',b_n')\}_{n=1}^{\infty}$, then the sequence $\{(a_n',b_n')\}_{n=L+1}^{\infty}$ is the sequence of continued fraction coefficients for a different discrete $m$-function, which we denote by $s_L$.  In this case, it holds that
\[
s_L(z)=\frac{p_L(z;\{a_n',b_n'\}_{n=1}^L)s(z)+q_L(z;\{a_n',b_n'\}_{n=1}^L)}{-a_L'(p_{L-1}(z;\{a_n',b_n'\}_{n=1}^{L-1})s(z)+q_{L-1}(z;\{a_n',b_n'\}_{n=1}^{L-1}))}
\]
(see \cite[Eqns. 3.2.23 and 3.7.23]{Rice}).

Let the continued fraction coefficients for $M$ be given by \eqref{perpar} with $(\alpha_k,\beta_k)=(a_p,b_p)$ and with $\{(a_j,b_j)\}_{j=1}^{\infty}$ doubly palindromic with first length $\ell\in\{1,2,\ldots,p-2\}$.  Let $m$ be the discrete $m$-function with purely periodic continued fraction coefficients $\{(a_j,b_j)\}_{j=1}^{\infty}$.  Then $m=M_k$ and so the matrix $T_1$ satisfies
\begin{equation}\label{step1}
m(z)=f_{T_1}(M(z)).
\end{equation}

Now let $m^{-}(z)$ be the discrete $m$-function with purely periodic continued fraction coefficients with one period given by
\[
\{(a_{p-1},b_p),(a_{p-2},b_{p-1}),\ldots,(a_1,b_2),(a_p,b_1)\}
\]
Write
\begin{align*}
a_1a_2\cdots a_p&=\overbrace{a_{1,1}a_{1,2}\cdots a_{1,2}a_{1,1}}^{\ell}a_{2,1}a_{2,2}\cdots a_{2,2}a_{2,1}\\
b_1b_2\cdots b_p&=\underbrace{b_{1,1}b_{1,2}\cdots b_{1,2}b_{1,1}}_{\ell+1}b_{2,1}b_{2,2}\cdots b_{2,2}b_{2,1}
\end{align*}
Then the sequences $\{a_j\}_{j=\ell+2}^{\infty}$ and $\{b_j\}_{j=\ell+2}^{\infty}$ begin
\[
a_{2,2}a_{2,3}\cdots a_{2,3}a_{2,2}a_{2,1}a_{1,1}\cdots a_{1,1}\cdots
\]
and
\[
b_{2,1}b_{2,2}\cdots b_{2,2}b_{2,1}b_{1,1}b_{1,2}\cdots b_{1,2}b_{1,1}\cdots
\]
respectively.  These coincide with the $p$-periodic sequences with one period given by
\[
\{a_{p-1},\ldots a_1,a_p\}\qquad \mbox{and}\qquad \{b_p,b_{p-1}\ldots,b_1\}
\]
respectively and these are the sequences of continued fraction coefficients for the discrete $m$-function $m^{-}$.  Therefore, $m^-=m_{\ell+1}$ and hence the matrix $T_2$ satisfies
\begin{equation}\label{step2}
m^-(z)=f_{T_2}(m(z))
\end{equation}

From Equation \eqref{step1}, we know that
\[
m(z)=\frac{p_k(z)M(z)+q_k(z)}{-\alpha_k(p_{k-1}(z)M(z)+q_{k-1}(z))},
\]
where $p_k(z)=p_k(z;\{\alpha_j,\beta_j\}_{j=1}^k)$ and similarly for $p_{k-1},q_k$, and $q_{k-1}$.  This shows that if $\widetilde{M}(z)$ is the second solution corresponding to $m(z)$, then
\[
\tilde{m}(z)=\frac{p_k(z)\widetilde{M}(z)+q_k(z)}{-\alpha_k(p_{k-1}(z)\widetilde{M}(z)+q_{k-1}(z))}=\frac{p_k(z)+q_k(z)\frac{1}{\widetilde{M}(z)}}{-\alpha_k(p_{k-1}(z)+q_{k-1}(z)\frac{1}{\widetilde{M}(z)})}
\]
Solving for $1/\widetilde{M}(z)$ gives
\[
\frac{1}{\widetilde{M}(z)}=\frac{p_k(z)+\tilde{m}(z)\alpha_kp_{k-1}(z)}{-(q_k(z)+\alpha_kq_{k-1}(z)\tilde{m}(z))}
\]
Now we will let $p_k^-(z)=p_k(z;\{\alpha_{k-j},\beta_{k-j+1}\}_{j=1}^k)$ (with $\alpha_0=\alpha_k$) and similarly for $p_{k-1}^-,q_k^-$, and $q_{k-1}^-$.  We invoke \cite[Lemma 5.2.3]{Rice} to rewrite the above as
\[
\frac{1}{\widetilde{M}(z)}=\frac{p_k^-(z)+\tilde{m}(z)\alpha_k^2q_{k}^-(z)}{-(\alpha_k^{-1}p_{k-1}^-(z)+\alpha_kq_{k-1}^-(z)\tilde{m}(z))}=\frac{p_k^-(z)+\tilde{m}(z)\alpha_k^2q_{k}^-(z)}{-\alpha_k(\alpha_k^{-2}p_{k-1}^-(z)+q_{k-1}^-(z)\tilde{m}(z))}
\]
Now we use \cite[Theorem 5.2.2]{Rice} to replace $\tilde{m}(z)$ by $(a_p^2m^-(z))^{-1}$ and use the fact that $a_p=\alpha_k$ to rewrite this as
\[
\frac{1}{\widetilde{M}(z)}=\frac{a_p^2p_k^-(z)m^-(z)+a_p^2q_{k}^-(z)}{-\alpha_k(p_{k-1}^-(z)m^-(z)+q_{k-1}^-(z))}
\]
and hence we see that the matrix $T_3$ satisfies
\begin{equation}\label{step3}
\frac{1}{\alpha_k^2\widetilde{M}(z)}=f_{T_3}(m^-(z))
\end{equation}
Combining Equations \eqref{step1}, \eqref{step2}, and \eqref{step3} proves this direction of Theorem \ref{mainthm}.

To prove the converse, assume that \eqref{maineq} holds for some $\ell\in\bbN$.  The first half of the proof shows that it will suffice to prove that $m^-=m_{\ell+1}$.  Notice that Equations \eqref{step1} and \eqref{step3} did not make use of the fact that the periodic portion of the continued fraction coefficients for $M$ is doubly palindromic (only Equation \eqref{step2} made use of this fact).  Therefore, when we write
\[
\frac{1}{\alpha_k^2\widetilde{M}(z)}=f_{T_3T_2T_1}(M(z))=f_{T_3}(f_{T_2}(f_{T_1}(M(z)))),
\]
we may replace
\[
f_{T_1}(M(z))\qquad\mbox{by}\qquad m(z)
\]
and we may replace
\[
\frac{1}{\alpha_k^2\widetilde{M}(z)}\qquad\mbox{by}\qquad f_{T_3}(m^-(z))
\]
Applying $f_{T_3^{-1}}$ to both sides (because $T_3$ has determinant $1$ (see \cite[Eqn. 3.2.22]{Rice})) shows
\[
m^-(z)=f_{T_2}(m(z)).
\]
Since the polynomials in $T_2$ are generated from the continued fraction coefficients that correspond to $m$, this last relation implies $m^-=m_{\ell+1}$ as desired.

\medskip

\noindent\textit{Remark.}  For future reference, we recall our observation that the relation $m^-=m_{\ell+1}$ is equivalent to the condition that the continued fraction coefficients of $m$ are doubly palindromic with first length $\ell$.  Since the continued fraction coefficients of $m^-$ and $m$ are both periodic with period $p$, this is also equivalent to $m=m^-_{p-\ell-1}$.

\section{Further Discussion}\label{further}

Given the relation in Theorem \ref{mainthm}, it is natural to wonder what happens if we reverse the role of $M$ and $\widetilde{M}$.  The implications of this switch are not obvious because only $M$ is assumed to be a discrete $m$-function.  If the continued fraction coefficients of $M$ are purely periodic, then $1/(a_p^2\widetilde{M})$ is a discrete $m$-function (which we denoted by $m^-$ in the previous section), but this is generally not the case.  Indeed, if retain the notation from the proof of Theorem \ref{mainthm} and set $k=1$, then we can write
\[
\widetilde{M}(z)=\frac{1}{\beta_1-z-\alpha_1^2\tilde{m}(z)}=\cfrac{1}{\beta_1-z-\cfrac{\alpha_1^2}{a_p^2m^-(z)}}
\]
At infinity, $m^-$ decays like $-1/z$ and hence $\widetilde{M}$ decays like $-[(1-\alpha_1^2/a_p^2)z]^{-1}$ at infinity.  This shows $1/(\alpha_1^2\widetilde{M})$ cannot be a discrete $m$-functions unless $\alpha_1=a_p$ and this choice of $\alpha_1$ would preserve periodicity.  This leads us to the following result.

\begin{theorem}\label{reverse}
Let $M$ be a discrete $m$-function with eventually periodic continued fraction coefficients as in \eqref{perpar} and satisfying \eqref{kp}.  Suppose there exists a sequence $\{(c_n,d_n)\}_{n=1}^{\infty}$ with each $c_n>0$ and each $d_n\in\bbR$ and an $L\in\bbN$ such that
\[
M(z)=\frac{p_{L}(z)\frac{1}{\alpha_k^2\widetilde{M}(z)}+q_{L}(z)}{-c_{L}(p_{L-1}(z)\frac{1}{\alpha_k^2\widetilde{M}(z)}+q_{L-1}(z))},
\]
where $p_{L}(z)=p_{L}(z;\{c_j,d_j\}_{j=1}^{L})$ and similarly for $q_{L}$, $p_{L-1}$, and $q_{L-1}$.  Then the continued fraction coefficients of $M$ are purely periodic and doubly palindromic.
\end{theorem}

To prove Theorem \ref{reverse}, we will need the following lemma.

\begin{lemma}\label{strips}
Let $f$ and $g$ be discrete $m$-functions.  Suppose that there exist continued fraction parameters $\{(\gamma_n,\delta_n)\}_{n=1}^{\infty}$ and $s\in\bbN$ so that
\[
f(z)=\frac{p_{s}(z)g(z)+q_{s}(z)}{-\gamma_{s}(p_{s-1}(z)g(z)+q_{s-1}(z))},
\]
where $p_{s}(z)=p_{s}(z;\{\gamma_j,\delta_j\}_{j=1}^{s})$ and similarly for $q_{s}$, $p_{s-1}$, and $q_{s-1}$. Then the first $s$ continued fraction coefficients of $g$ are $\{(\gamma_j,\delta_j)\}_{j=1}^s$.
\end{lemma}

\begin{proof}
Through algebraic manipulation, we see that we can rewrite the hypotheses of the lemma as
\[
g(z)=\cfrac{1}{\delta_1-z-\cfrac{\gamma_1^2}{\delta_2-z-\cfrac{\gamma_2^2}{\ddots-\cfrac{\gamma_{s-1}^2}{\delta_s-z-\gamma_s^2f(z)}}}}
\]
If $(x_1,y_1)$ is the first continued fraction coefficient for $g$, then we can also write
\[
g(z)=\cfrac{1}{y_1-z-x_1^2g_1(z)}
\]
(recalling the notation $g_1$ from the beginning of Section \ref{mainproof}).  Since $f$ and $g_1$ are both discrete $m$-functions, then by comparing asymptotics at infinity, we deduce that $x_1=\gamma_1$ and $y_1=\delta_1$.  Applying this same reasoning to $g_1$ and proceeding inductively yields the desired relation.
\end{proof}

\begin{proof}[Proof of Theorem \ref{reverse}]
Let
\[
T_4=\begin{pmatrix}
p_{L}(z;\{c_j,d_j\}_{j=1}^{L}) & q_{L}(z;\{c_j,d_j\}_{j=1}^{L})\\
-c_{L}p_{L-1}(z;\{c_j,d_j\}_{j=1}^{L}) & -c_{L}q_{L-1}(z;\{c_j,d_j\}_{j=1}^{L})
\end{pmatrix}
\]
so that our assumptions can be rewritten as $M=f_{T_4}((\alpha_k^2\widetilde{M})^{-1})$.  Using the notation from the proof of Theorem \ref{mainthm}, we may then conclude that
\[
m(z)=f_{T_1T_4T_3}(m^-(z)).
\]
The matrix $T_1T_4T_3$ is a conjugated transfer matrix (see the Remark after Theorem \ref{mainthm}) and so we may apply Lemma \ref{strips} to conclude that $m=m^-_N$ for some $N\in\bbN$.  We have already observed that this is equivalent to the condition that the sequence of continued fraction coefficients for $m$ is doubly palindromic (and hence that for $M$ is eventually doubly palindromic).  Thus, we may apply Theorem \ref{mainthm} to write $M=f_{T_4T_3T_2T_1}(M(z))$.  Now we may iterate this relation and apply Lemma \ref{strips} again to obtain our desired conclusion.
\end{proof}




\vspace{7mm}





\begin{thebibliography}{99}






\bibitem{Akh} N. I. Akhiezer, {\em The classical moment problem and some related questions in analysis}, Translated by N. Kemmer Hafner Publishing Co., New York 1965.

\bibitem{Jungle} E. Burger, {\em Exploring the number jungle: a journey into Diophantine analysis}, Student Mathematical Library, 8. American Mathematical Society, Providence, RI, 2000.

\bibitem{Burger05} E. Burger, {\em A tail of two palindromes},  Amer. Math. Monthly 112 (2005), no. 4, 311--321.

\bibitem{DMS} M. Derevyagin, A. Minenkova, and N. Sun, {\em A theorem of Joseph-Alfred Serret and its relation to perfect quantum state transfer}, Expo. Math. 39 (2021), no. 3, 480--499.

\bibitem{Rice} B. Simon, {\em Szeg\H{o}'s Theorem and its descendants. Spectral theory for $L^2$ perturbations of orthogonal polynomials}, M. B. Porter Lectures. Princeton University Press, Princeton, NJ, 2011.

\bibitem{Wall} H. S. Wall, {\em On some criteria of Carleman for the complete convergence of a $J$-fraction},  Bull. Amer. Math. Soc. 54 (1948), 528--532.














\end{thebibliography}
\end{document}